\documentclass[12pt]{article}
\usepackage{a4,amsmath,amsthm,amssymb,enumerate,eucal,graphicx,subfig,xcolor}
\title{\textbf{Multiple Petersen subdivisions\\in permutation
    graphs}\footnote{Supported by project MEB 021115 of the Czech
    Ministry of Education and PHC Barrande 24444XD of the French MAE.
    This work falls within the scope of the \textsc{LEA STRUCO}.}}%
\author{Tom\'{a}\v{s} Kaiser$^{\:1}$\\
  Jean-S\'{e}bastien Sereni$^{\:2}$\\
  Zelealem Yilma$^{\:3}$} \date{}

\newtheorem{theorem}{Theorem}
\newtheorem{lemma}[theorem]{Lemma}
\newtheorem{proposition}[theorem]{Proposition}

\newtheorem{corollary}[theorem]{Corollary}

\newcommand\size[1] {\left|{#1}\right|}
\newcommand\Set[2] {\left\{{#1}:\,{#2}\right\}}
\newcommand\Setx[1] {\left\{{#1}\right\}}

\newcommand{\isom}{\simeq}
\newcommand{\match}[2]{#1\{#2\}}
\newcommand{\ol}{\overline}

\newcommand{\fig}[1]{\includegraphics[page=#1]{pp-figs.pdf}}%
\newcommand{\sfig}[2]{\subfloat[#2]{\fig{#1}}}%

\newcommand{\sfigtop}[2]{\newbox{\pic}\sbox{\pic}{\fig{#1}}%
  \subfloat[#2]{\vbox to\ht\base{\hbox to\wd\pic{\usebox\pic}}}}
\newcommand{\hf}{\hspace*{0pt}\hspace{\fill}\hspace*{0pt}}

\begin{document}
\maketitle
\footnotetext[1]{Department of Mathematics and Institute for
  Theoretical Computer Science, University of West Bohemia,
  Univerzitn\'{\i}~8, 306~14~Plze\v{n}, Czech Republic. E-mail:
  \texttt{kaisert@kma.zcu.cz}. Supported by project P202/12/G061 of
  the Czech Science Foundation.}%
\footnotetext[2]{CNRS (LIAFA, Universit\'{e} Diderot), Paris, France.
  E-mail: \texttt{sereni@kam.mff.cuni.cz}. This author's work was
  partially supported by the French \emph{Agence Nationale de la
    Recherche} under reference \textsc{anr 10 jcjc 0204 01}.}%
\footnotetext[3]{LIAFA, Universit\'{e} Denis Diderot (Paris 7), 175
  Rue du Chevaleret, 75013 Paris, France. E-mail:
  \texttt{Zelealem.Yilma@liafa.jussieu.fr}.  This author's work was
  partially supported by the French \emph{Agence Nationale de la
    Recherche} under reference \textsc{anr 10 jcjc 0204 01}.}

\begin{abstract}
  A permutation graph is a cubic graph admitting a 1-factor $M$ whose
  complement consists of two chordless cycles. Extending results of
  Ellingham and of Goldwasser and Zhang, we prove that if $e$ is an
  edge of $M$ such that every 4-cycle containing an edge of $M$
  contains $e$, then $e$ is contained in a subdivision of the Petersen
  graph of a special type. In particular, if the graph is cyclically
  5-edge-connected, then every edge of $M$ is contained in such a
  subdivision. Our proof is based on a characterization of cographs in
  terms of twin vertices. We infer a linear lower bound on the number
  of Petersen subdivisions in a permutation graph with no 4-cycles,
  and give a construction showing that this lower bound is tight up to
  a constant factor.
\end{abstract}

%%%%%%%%%%%%%%%%%%%%%%%%%%%%%%%%%%%%%%%%%%%%%%%%%%%%%%%%%%%%%%%%%%%%%%

\section{Introduction}
\label{sec:introduction}

A special case of Tutte's 4-flow conjecture~\cite{Tut:algebraic}
states that every bridgeless cubic graph with no minor isomorphic to
the Petersen graph is 3-edge-colourable. Before this special case was
shown to be true by Robertson et al. (cf.~\cite{Tho:recent}), one of
the classes of cubic graphs for which the conjecture was known to hold
was the class of permutation graphs --- i.e., graphs with a 2-factor
consisting of two chordless cycles. Indeed, by a result of
Ellingham~\cite{Ell:petersen}, every permutation graph is either
Hamiltonian --- and hence 3-edge-colourable --- or contains a
subdivision of the Petersen graph. To state his theorem more
precisely, we introduce some terminology.

Rephrasing the above definition, a cubic graph $G$ is a
\emph{permutation graph} if it contains a perfect matching $M$ such
that $G-E(M)$ is the disjoint union of two cycles, none of which has a
chord in $G$. A perfect matching $M$ with this property is called a
\emph{distinguished matching} in $G$. For brevity, if $G$ is a
permutation graph with a distinguished matching $M$, then the
pair $(G,M)$ is referred to as a \emph{marked permutation graph}.

We let $P_{10}$ be the Petersen graph. Given a distinguished
matching $M$ in $G$, an \emph{$M$-copy of $P_{10}$} is a subgraph $G'$
of $G$ isomorphic to a subdivision of $P_{10}$ and composed of the
two cycles of $G-E(M)$ together with five edges of
$M$. Following Goldwasser and Zhang~\cite{GZ:permutation}, an $M$-copy of $P_{10}$ is also
referred to as an $M$-$P_{10}$. Furthermore, an \emph{$M$-copy of the
  4-cycle $C_4$} (or an \emph{$M$-$C_4$}) is a $4$-cycle in $G$ using
two edges of $M$.

The proof of Ellingham's result implies that if a marked permutation
graph $(G,M)$ contains no $M$-$P_{10}$, then it contains an $M$-$C_4$
(and is therefore Hamiltonian). Goldwasser and
Zhang~\cite{GZ:permutation} obtained a slight strengthening:
\begin{theorem}\label{t:zhang}
  If $(G,M)$ is a marked permutation graph, then $G$ contains either
  two $M$-copies of $C_4$, or an $M$-copy of $P_{10}$.
\end{theorem}

Lai and Zhang~\cite{LZ:hamilton} studied permutation graphs satisfying
a certain minimality condition and proved that in a sense, they
contain `many' subdivisions of the Petersen graph.

The main result of this note is the following generalization of
Theorem~\ref{t:zhang}.
\begin{theorem}\label{t:main}
  Let $(G,M)$ be a marked permutation graph on at least six vertices
  and let $e\in E(M)$. If $e$ is contained in every $M$-$C_4$ of $G$, then
  $e$ is contained in an $M$-copy of $P_{10}$.
\end{theorem}

Theorem~\ref{t:main} is established in Section~\ref{sec:proof}. The
proof is based on a relation between $M$-copies of $P_{10}$ in
permutation graphs and induced paths in a related class of graphs.

Of particular interest is the corollary for cyclically
5-edge-connected graphs, that is, graphs containing no edge-cut of
size at most $4$ whose removal leaves at least two non-tree
components.
\begin{corollary}\label{cor:5-conn}
  Every edge of a cyclically 5-edge-connected marked permutation graph
  $(G,M)$ is contained in an $M$-$P_{10}$.
\end{corollary}
The class of cyclically 5-edge-connected permutation graphs is richer
than one might expect. Indeed, it had been
conjectured~\cite{Zha:integer} that every cyclically 5-edge-connected
permutation graph is 3-edge-colourable, but this conjecture has been
recently disproved~\cite[Observation~4.2]{BGHM:generation}.

Theorem~\ref{t:main} readily implies a lower bound on the number of
$M$-copies of the Petersen graph in a marked permutation graph $(G,M)$
such that $G$ contains no $M$-$C_4$ and has $n$ vertices. We improve
this lower bound in Section~\ref{sec:counting}. We also show
that the bounds (which are linear in $n$) are optimal up to a constant
factor.

We close this section with some terminology. If $G$ is a graph
and $X\subseteq V(G)$, then $G[X]$ is the induced subgraph of $G$ on
$X$. The set of all neighbours of a vertex $v$ of $G$ is denoted by
$N_G(v)$.

%%%%%%%%%%%%%%%%%%%%%%%%%%%%%%%%%%%%%%%%%%%%%%%%%%%%%%%%%%%%%%%%%%%%%%

\section{Proof of Theorem~\ref{t:main}}
\label{sec:proof}
Let $(G,M)$ be a marked permutation graph. If $v\in V(G)$, then we
write $v'$ for the neighbour of $v$ in $M$ (which we call the
\emph{friend} of $v$). We extend this notation to arbitrary sets of
vertices of $G$: if $X\subseteq V(G)$, then we set
\begin{equation*}
  X' = \Set{v'}{v\in X}.
\end{equation*}
Let $A$ be the vertex set of one component of $G-E(M)$. Thus,
$A'$ is the vertex set of the other component and both $G[A]$ and
$G[A']$ are chordless cycles.

In this section, we prove Theorem~\ref{t:main}. Fix an edge $e$ of the
matching $M$. Let $a$ and $a'$ be its end-vertices.  We also choose an
orientation for each of the cycles $G[A]$ and $G[A']$. All these will
be fixed throughout this section.

If $X\subseteq A$, then $\match G X$ is the spanning
subgraph of $G$ obtained by adding to $G-E(M)$ all the edges $vv'$, where $v\in
X$. In expressions such as $\match G {\Setx{a,b}}$, we
omit one pair of set brackets, and write just $\match G {a,b}$.

The \emph{auxiliary graph} $H_a$ (with respect to the vertex $a$) is
defined as follows. The vertex set of $H_a$ is $A-\Setx{a}$. Two
vertices $x$ and $y$ of $H_a$ are adjacent in $H_a$ whenever the
cyclic order of $a$, $x$ and $y$ on $G[A]$ is $axy$ and the cyclic
order of their friends on $G[A']$ is $a'y'x'$.

\begin{figure}
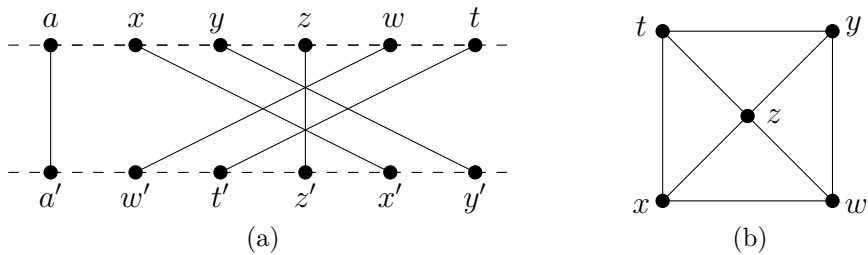

  \centering
  \hf\sfig2{}\hf\sfig3{}\hf
  \caption{(a) The standard drawing of the graph $G$. (b) The
    corresponding graph $H_a$.}
  \label{fig:cross}
\end{figure}

Alternatively, consider the following standard procedure, illustrated
in Figure~\ref{fig:cross}. Arrange the vertices of $A$ on a horizontal
line in the plane, starting on the left with $a$ and continuing along
the cycle $G[A]$ according to the fixed orientation. Place the
vertices of $A'$ on another horizontal line, putting $a'$ leftmost and
continuing in accordance with the orientation of $G[A']$. Join each
vertex $z\in A$ with its friend by a straight line segment. The segment $aa'$
is not crossed by any other segment, and for
$x,y\in A-\Setx a$, the segments $xx'$ and $yy'$ cross each other if
and only if $x$ and $y$ are adjacent in $H_a$.
Thus, $H_a$ can be directly read off the resulting figure,
which is called the \emph{standard drawing} of $G$.

A similar construction, without fixing the vertex $a$,
gives rise to a class of graphs also called `permutation graphs'
(see~\cite{BLS:graph}). In this paper, we only use this term as defined
in Section~\ref{sec:introduction}.

The following lemma provides a link between induced paths in $H_a$ and
$M$-copies of $P_{10}$ in $G$.
\begin{lemma}\label{l:path}
  Suppose that $H_a$ contains an induced path $xyzw$ on 4
  vertices. Then $\match G {a,x,y,z,w}$ is an $M$-$P_{10}$ in $G$.
\end{lemma}
\begin{proof}
  Let $P$ be the path $xyzw$ in $H_a$. Since $xy\in E(P)$, the edges
  $xx'$ and $yy'$ cross. By symmetry, we may assume that $x\in aCy$
  and $y'\in a'C'x'$.
  First, note that $z\notin yCa$. Otherwise, as $zz'$ crosses
  $yy'$, it would follow that $zz'$ also crosses $xx'$, which contradicts the assumption
  that $x$ and $z$ are not adjacent in $G$.
%  $xz\notin E(P)$.
  We now consider two cases, regarding whether or not $z\in aCx$.

  Case 1: $z\in aCx$. Since the edges $zz'$ and $xx'$ do not cross,
  $z'\in a'C'x'$; moreover, since $zz'$ and $yy'$ cross, it follows
  that $z'\in y'C'x'$.

  We assert that $w\in zCx$. Suppose that this is not the case. If
  $w\in aCz$, then $ww'$ cannot cross $zz'$ without crossing $yy'$,
  contradicting the fact that $zw\in E(P)$ and $yw\notin E(P)$. If
  $w\in xCy$, then $ww'$ crosses $xx'$ or $yy'$ regardless of the
  position of $w'$, which results in a similar contradiction. Finally,
  if $w\in yCa$, then $ww'$ cannot cross $zz'$ without crossing $xx'$.

  Thus, we have shown that $w\in zCx$, which implies that $w'\in
  a'C'y'$, as $ww'$ and $yy'$ do not cross. Summing up, $\match G
  {a,x,y,z,w}$ is precisely as in Figure~\ref{fig:case} and
  constitutes an $M$-copy of the Petersen graph.

  \begin{figure}
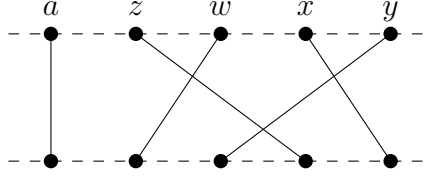

    \centering
    \fig1
    \caption{The graph $\match G {a,x,y,z,w}$ in Case 1 of the proof
      of Lemma~\ref{l:path}.}
    \label{fig:case}
  \end{figure}

  Case 2: $z\notin aCx$. Then, $z\in xCy$. Since $zz'$ and $xx'$ do
  not cross, $z'\in x'C'a'$. As $ww'$ crosses $zz'$ but none of $xx'$
  and $yy'$, the only possibility is that $w\in yCa$ and $w'\in
  x'C'z'$, which again produces an $M$-$P_{10}$.
\end{proof}

By Lemma~\ref{l:path}, if there is no $M$-copy of $P_{10}$ containing
$aa'$ in $G$, then $H_a$ contains no induced path on 4 vertices. Such
graphs are known as \emph{cographs} or \emph{$P_4$-free} graphs. There
are various equivalent ways to describe them, summarized in the
survey~\cite[Theorem~11.3.3]{BLS:graph} by Brandst\"{a}dt, Le and
Spinrad.  We use the characterisation that involves pairs of twin
vertices.  (Two vertices $x$ and $y$ of a graph $H$ are \emph{twins}
if $N_H(x) = N_H(y)$.)

\begin{theorem}\label{t:p4free}
  A graph $G$ is $P_4$-free if and only if every induced subgraph of
  $G$ with at least two vertices contains a pair of twins.
\end{theorem}

To be able to use Lemma~\ref{l:path} in conjunction with
Theorem~\ref{t:p4free}, we need to interpret twin pairs of $H_a$ in
terms of $G$.

\begin{lemma}\label{l:twin}
  Let $x$ and $y$ be twins in $H_a$. Let $Q'$ be the path in $G$
  defined by
  \begin{equation*}
    Q' =
    \begin{cases}
      x'C'y' & \text{if $xy\notin E(H_a)$,}\\
      y'C'x' & \text{otherwise.}
    \end{cases}
  \end{equation*}
  Then $M$ matches the vertices of the path $xCy$ to those of $Q'$
  and \emph{vice versa}.
\end{lemma}
\begin{proof}
  Suppose, on the contrary, that the statement does not hold. By
  symmetry, we may assume that $x\in aCy$, and that $M$ contains an
  edge $ww'$ with $w \in V(xCy)$ and $w'\notin V(Q')$. We assert that
  $ww'$ crosses exactly one edge from $\Setx{xx',yy'}$. To prove this,
  we consider two cases according to whether or not $xx'$ and $yy'$
  cross. If they do not cross, then $ww'$ crosses only $xx'$ (if
  $w'\in V(a'C'x')$) or only $yy'$ (if $w'\in V(y'C'a')$). Otherwise,
  $ww'$ crosses only $xx'$ (if $w'\in V(a'C'y')$) or only $yy'$ (if
  $w'\in V(x'C'a')$). In each case, we obtain a contradiction with the
  assumption that $x$ and $y$ are twins in $H_a$.
\end{proof}

\begin{figure}
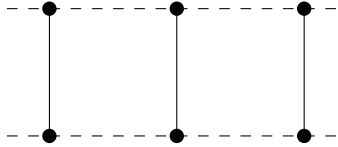

  \centering
  \hf\fig4\hf
  \caption{The triangular prism $G$ with the unique 1-factor $M$ such
    that $(G,M)$ is a marked permutation graph.}
  \label{fig:prism}
\end{figure}

We now prove Theorem~\ref{t:main}, proceeding by induction on the
number of vertices of $G$. The base case is the triangular prism, the
unique permutation graph on $6$ vertices (Figure~\ref{fig:prism}), for
which the theorem is trivially true since $e$ cannot be contained in
every $M$-$C_4$ of $G$. Therefore, we assume that $G$ has at least $8$
vertices and that every $M$-copy of $C_4$ in $G$ contains the edge
$aa'$. 

Suppose first that $azz'a'$ is such an $M$-copy of $C_4$. Let $G_0$ be
the cubic graph obtained by removing the edge $zz'$ and suppressing
the resulting degree 2 vertices $z$ and $z'$. Set
$M_0=M\setminus\Setx{zz'}$. All $M_0$-copies of $C_4$ created by this
operation contain the edge $aa'$.  Therefore, regardless of whether or
not $G_0$ contains an $M_0$-$C_4$, the induction hypothesis implies
that $aa'$ is contained in an $M_0$-copy of $P_{10}$.  This yields an
$M$-copy of $P_{10}$ in $G$ containing $aa'$, as required.

Consequently, it may be assumed that $G$ does not contain any
$M$-$C_4$. Assume that $H_a$ contains no pair of twin
vertices. Theorem~\ref{t:p4free} implies that $H_a$ is not
$P_4$-free. Let $X$ be a subset of $V(H_a)$ of size 4 such that
$H_a[X] \isom P_4$. By Lemma~\ref{l:path}, $\match G {X\cup\Setx{a}}$
is an $M$-copy of $P_{10}$, and the sought conclusion follows.

Thus, we may assume that $H_a$ contains twin vertices $x$ and
$y$. Without loss of generality, $x$ belongs to $aCy$.

Let the path $Q'$ be defined as in Lemma~\ref{l:twin}. Thus, since $x$
and $y$ are twins in $H_a$, vertices of the path $xCy$ are only
adjacent in $M$ to vertices of $Q'$ and \emph{vice versa}. We
transform $G$ into another cubic graph $G_1$ by removing all vertices
that are not contained in $\Setx{a,a'} \cup V(xCy\cup Q')$ and adding
the edges $ax$, $ay$, $a'x'$ and $a'y'$ (if they are not present
yet). Let $M_1$ be the perfect matching of $G_1$ consisting of all the
edges of $M$ contained in $G_1$. Note that although the transformation
may create $M$-copies of $C_4$ not present in $G$, the edge $aa'$ is
contained in every $M_1$-copy of $C_4$ in $G_1$. Furthermore, the path
$yCx$ in $G$ must have some internal vertices other than $a$, since
otherwise $G$ would contain an $M$-$C_4$, namely $yaa'y'$. Thus, $G_1$
has fewer vertices than $G$. The induction hypothesis implies that
$aa'$ is contained in an $M_1$-copy of $P_{10}$ in $G_1$, and
therefore also in $G$.

%%%%%%%%%%%%%%%%%%%%%%%%%%%%%%%%%%%%%%%%%%%%%%%%%%%%%%%%%%%%%%%%%%%%%%

\section{Counting the Petersen copies}
\label{sec:counting}

Turning to the quantitative side of the question studied in
Section~\ref{sec:proof}, we now derive from Theorem~\ref{t:main} a
lower bound on the number of $M$-copies of $P_{10}$ in a permutation
graph with no $M$-$C_4$. The bound is linear in the order of the
graph. We give a construction showing that this lower bound is tight
up to a constant factor.

Throughout this section, $(G,M)$ is a marked permutation graph with
vertex set $A\cup A'$ just like in Section~\ref{sec:proof}.

We will need two lemmas, the second of which we find to be of interest
in its own right. The first lemma is an observation on auxiliary
graphs which follows readily from the definition; its proof is
omitted.
\begin{lemma}\label{l:redrawing}
  Let $a,b \in A$. Then the following hold for each $x,y\in
  A-\Setx{a,b}$:
  \begin{enumerate}[\quad(i)]
  \item $ax \in H_b$ if and only if $bx \in H_a$,
  \item $xy \in H_b$ if and only if $\size{\Setx{bx, by, xy} \cap H_a}
    \in \Setx{1,3}$.
  \end{enumerate}
\end{lemma}

\begin{lemma}\label{l:replace}
  Let $a,b \in A$. One of the following conditions holds:
  \begin{itemize}
  \item there is some $M$-$P_{10}$ in $G$ containing both $aa'$ and
    $bb'$, or
  \item for any $F \subset A$ with $\size F = 4$ and $\Setx{a,b}\cap F
    = \emptyset$, it holds that $\match G {F\cup\Setx a} \isom P_{10}$
    if and only if $\match G {F\cup\Setx b} \isom P_{10}$.
  \end{itemize}
\end{lemma}
\begin{proof}
  Assume that there exists no $M$-$P_{10}$ containing both $aa'$ and
  $bb'$.  By Lemma~\ref{l:path}, it is sufficient to show that a set
  $\Setx{u,w,x,y} \subseteq A \setminus \Setx{a,b}$ induces a path of
  length $4$ in $H_a$ if and only if it induces a path of length $4$
  in $H_b$.

  Let $U_1 = N_{H_a}(b) = N_{H_b}(a)$ and $U_2 = A \setminus
  (U_1\cup\Setx{a,b})$. Lemma~\ref{l:path} implies that in the auxiliary graph
  $H_a$, there is no induced path of length $4$ containing the vertex
  $b$. Therefore,
  \begin{enumerate}[\quad(i)]
  \item if $x,y \in U_1$, $z \in U_2$, and $xy \notin H_a$, then $xz
    \in H_a$ if and only if $yz \in H_a$, and
  \item if $x,y \in U_2$, $z \in U_1$, and $xy \in H_a$,
    then $xz \in H_a$ if and only if $yz \in H_a$.
  \end{enumerate}
  Hence, if $uwxy$ is an induced path of length $4$ in $H_a$, then
  $\Setx{u,w,x,y} \cap U_1 \in \Setx{\Setx{u,w,x,y},\Setx{w,x},
    \emptyset}$.  By Lemma~\ref{l:redrawing}, it follows that
  $\Setx{u,w,x,y}$ induces a path of length $4$ in $H_b$ as well.
  More precisely, this path is $uwxy$ if $\Setx{u,w,x,y} \cap U_1 \in
  \Setx{\Setx{u,w,x,y},\emptyset}$, and $uxwy$ if $\Setx{u,w,x,y} \cap
  U_1 = \Setx{w,x}$.  The conclusion follows by symmetry of the roles
  played by $a$ and $b$.
\end{proof}

We can now prove the aforementioned lower bound.
\begin{proposition}\label{p:lower}
  If $(G,M)$ is a marked permutation graph with $n \geq 40$ vertices
  and no $M$-$C_4$, then $(G,M)$ contains at least $n/2-4$
  $M$-copies of the Petersen graph.
\end{proposition}
\begin{proof}
  If each edge of $M$ is contained in at least 5 $M$-copies of
  $P_{10}$, the total number of copies is at least $(5n/2)/5 =
  n/2$. Hence, we may assume that there exists $x\in\{1,2,3,4\}$
  and an edge $e \in E(M)$ that is contained in only $x$ $M$-copies of $P_{10}$.
  Let $\mathcal C$ be the set of these copies.

  At least $n/2-4x-1$ edges of $M$ are not contained in any
  $M$-$P_{10}$ containing $e$. By Lemma~\ref{l:replace}, if we replace
  $e$ by any such edge in any $M$-$P_{10}$ from $\mathcal C$, we
  obtain an $M$-$P_{10}$ again. These replacements yield $x(n/2-4x-1)$
  distinct $M$-copies of $P_{10}$. Thus, in total, $(G,M)$ contains at
  least $x(n/2-4x)$ distinct $M$-copies of $P_{10}$.
  Minimizing this expression over $x\in\{1,2,3,4\}$ and
  using the assumption that $n\geq 40$, we deduce that the
  number of copies is at least $n/2-4$, as asserted.
\end{proof}

We now construct a family of marked permutation graphs $(G_k,M_k)$
showing that the linear estimate in Proposition~\ref{p:lower} is tight
up to a constant factor. The graph $G_k$ has $6k+14$ vertices,
contains no $M_k$-$C_4$, and the number of $M_k$-copies of the
Petersen graph in $G_k$ is only $6k+6$. (We note that graphs with a
somewhat similar structure are constructed
in~\cite[Section~3]{GZ:permutation}.)

\begin{figure}
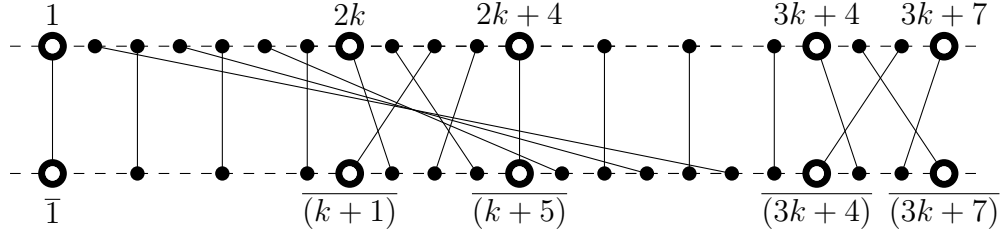

  \centering
  \hf\fig5\hf
  \caption{The marked permutation graph $(G_4,M_4)$. Labels are given
    only for the circled vertices.}
  \label{fig:linear}
\end{figure}

For $k=4$, the graph $(G_k,M_k)$ is shown in
Figure~\ref{fig:linear}. We now give a formal definition and
determine the number of $M_k$-copies of $P_{10}$.

Let $A = \Setx{1,2,\dots,3k+7}$ and $A' = \Setx{\ol 1,\ol
  2,\dots,\ol{3k+7}}$. The vertex set of $G_k$ is $A\cup A'$. On each
of $A$ and $A'$, we consider the standard linear order (in particular,
$\ol 1 < \ol 2 < \dots < \ol{3k+7}$). As in Section~\ref{sec:proof},
we write $i'$ for the neighbour in $M_k$ of a vertex $i\in A$. Thus,
$i' = \overline j$ for a suitable $j$.

Let
\begin{align*}
  E_1 &= \Set{(2i-1)\ol i}{1\leq i \leq k} \cup
  \Set{(2k+i+3)\ol{(k+2i+3)}}{1\leq i \leq k},\\
  E_2 &= \Set{(2i)\ol{(3k+4-2i)}}{1\leq i \leq k-1},\\
  E_3 &=
  \{(2k)\ol{(k+2)},(2k+1)\ol{(k+4)},(2k+2)\ol{(k+1)},(2k+3)\ol{(k+3)},\\
  &\qquad
  (3k+4)\ol{(3k+5)},(3k+5)\ol{(3k+7)},(3k+6)\ol{(3k+4)},(3k+7)\ol{(3k+6)}\}.
\end{align*}
Edges in $E_1$, $E_2$ and $E_3$ will be called \emph{vertical},
\emph{skew} and \emph{special}, respectively. Moreover, the first four
and the last four edges in $E_3$ are two \emph{groups of special edges}.

\begin{proposition}\label{prop:last}
  The marked permutation graph $(G_k,M_k)$ contains exactly $6k+6$
  $M$-copies of the Petersen graph.
\end{proposition}
\begin{proof}
  Each of the groups of special edges forms an $M_k$-$P_{10}$ with
  each of the remaining $3k+3$ edges of $M_k$. We prove that besides
  these $6k+6$ copies, there are no other $M_k$-copies of $P_{10}$ in
  $G_k$.

  For $X\subseteq M_k$, we let $G_X$ be the graph obtained
  from $G_k-M_k$ by adding the edges in $X$ and suppressing the degree
  $2$ vertices.

  Let $X$ be a subset of $M_k$ that contains no group of special
  edges. Suppose that $\match{G_k} X$ is isomorphic to the Petersen
  graph. To obtain a contradiction, we show that $\match{G_k} X$ contains a
  $4$-cycle.

  First of all, if $X$ contains a special edge, then it contains no
  other special edge from the same group. Indeed, a quick case
  analysis shows that if $Y$ consists of any two or three special
  edges in the same group, then $G_Y$ contains a $Y$-$C_4$.

  Thus, for the purposes of our argument, special edges behave just
  like vertical ones. We assert next that $X$ contains at most one skew
  edge. Let $j_1j'_1$ and $j_2j'_2$ be skew edges with
  $j_1 < j_2$ and $j_1+j_2$ maximum among the skew edges in $X$.

  Observe that $X$ contains no vertical edge $ii'$ with $i > j_2$ and
  $i' < j'_2$. Indeed, if there is only one such edge, then it forms
  an $X$-$C_4$ in $G_X$ together with $j_2j'_2$, while if there are at
  least two such edges, then an $X$-$C_4$ is obtained from a
  consecutive pair among them.

  By a similar argument, $X$ contains neither any vertical edge $ii'$
  with $j_1 < i < j_2$, nor any vertical edge $ii'$ with $j'_2 < i' <
  j'_1$. It follows that $j_1j'_1$ and $j_2j'_2$ are contained in an
  $X$-$C_4$ in $G_X$, a contradiction which proves that there is at
  most one skew edge in $X$.

  Consequently, $X$ contains a set $Y$ of at least four edges that are
  vertical or special, as $\size X = 5$. Further, $\size Y \neq 5$, so
  there are exactly four $Y$-copies of $C_4$ in $G_Y$. Only at most
  two of these will be affected by the addition of the fifth edge of
  $X$. Thus, an $X$-$C_4$ persists in $G_X$, a contradiction. The
  proof is complete.
\end{proof}

While the graphs constructed in the proof of
Proposition~\ref{prop:last} are $C_4$-free, they are not cyclically
$5$-edge connected. A slight modification of the construction ensures
this stronger property, but makes the discussion somewhat more
complicated. For this reason, we only described the simpler version.
%%%%%%%%%%%%%%%%%%%%%%%%%%%%%%%%%%%%%%%%%%%%%%%%%%%%%%%%%%%%%%%%%%%%%%

\end{document}